\documentclass[preprintnumbers,amsmath,amssymb,showkeys,showpacs]{revtex4}
\usepackage{lineno,hyperref}
\usepackage{amsfonts}
\usepackage{amsmath}
\usepackage{amssymb}
\usepackage{mathrsfs}
\usepackage[english,activeacute]{babel}
\usepackage[latin1]{inputenc}
\usepackage{graphicx}
\usepackage{subfigure}
\usepackage{float}

 \usepackage{color}

\numberwithin{equation}{section}
\newtheorem{lemma}{Lemma}
\newtheorem{corollary}{Corollary}
\newtheorem{theorem}{Theorem}
\newenvironment{proof}[1][Proof]{\begin{trivlist}
\item[\hskip \labelsep {\bfseries #1}]}{\end{trivlist}}

\begin{document}



\title{Elliptic solutions and solitary waves  of a higher order KdV--BBM  long wave equation}

\author{Stefan C. Mancas}
\email[Electronic address for correspondence: ]{mancass@erau.edu}

 \author{Ronald Adams}
  \email{adamsr25@erau.edu}

\affiliation{Department of Mathematics, Embry-Riddle Aeronautical University,\\ Daytona Beach, FL. 32114-3900, U.S.A.}

\begin{abstract}
We provide conditions for existence of  hyperbolic, unbounded periodic and  elliptic solutions in terms of Weierstrass $\wp$ functions of  both third and fifth-order KdV--BBM (Korteweg-de Vries--Benjamin, Bona \& Mahony)  regularized long wave equation.  An  analysis for the initial value problem  is developed together with  a local and global well-posedness theory for the third-order KdV--BBM  equation. Traveling wave reduction is used together with zero boundary conditions to yield solitons and periodic unbounded solutions, while for  nonzero boundary conditions we find solutions in terms of  Weierstrass elliptic $\wp$ functions. For the fifth-order KdV--BBM  equation we show that a parameter $\gamma=\frac {1}{12}$, for which the equation has a Hamiltonian, represents a restriction for which there are  constraint curves that never  intersect a region of unbounded solitary waves, which in turn shows that only dark or bright solitons and  no  unbounded solutions exist. Motivated by the lack of a Hamiltonian structure for $\gamma\neq\frac{1}{12}$ we develop $H^k$ bounds, and we show for the non Hamiltonian system that dark and bright  solitons coexist together with unbounded periodic solutions. For nonzero boundary conditions, due to the complexity of the nonlinear algebraic system of  coefficients of the elliptic equation we  construct Weierstrass solutions for a particular set of parameters only.
\end{abstract}



\maketitle


\section{Introduction}\label{sec1}
In a recent paper \cite{Carv}, the authors have derived a  second-order  mathematical description of long-crested water waves propagating  in one direction, which is analogous to a first-order approximation of a KdV--BBM-type equation, which  has the advantage that its solutions are expected to  be more accurate on a much longer time scale \cite{BCPS,Carv}.

The wave motion of crested waves propagate in the direction $x$, the bottom is flat with undisturbed depth $h_0$, the undisturbed and dependent variable is  $u(x; t) = h(x; t)- h_0$, where $h(x; t)$ is
the height of the water column at the horizontal location $x$ on the bottom at time $t$. The fact that the waves are long-crested is based on the assumption that the wave amplitudes and wavelengths are small and large as compared to the depth $h_0$ of the flat bottom.  If $A$ is the amplitude of the wave with wave length $\lambda$  then $\alpha=A/h_0 \ll 1 $, $\beta =h_0^2/\lambda^2\ll1$. The Stokes\rq{} number $S=\alpha/\beta \approx 1$ means that the nonlinear and dispersive  effects are balanced. The parameter $1/\alpha \approx 1/\beta$  is the so-called Boussinesq time for which models like BBM and KdV are known to provide good approximations of the unidirectional solutions of the full water wave problem \cite{1a,15a,17a,18a,29a,30a}. In ocean wave modeling, waves need to be followed on a time longer than Bousinesq time scale, and hence a higher order approximation to the water wave problem would be valid on the scale $1/\beta^2$. 

In their description \cite{BCPS} dissipation and surface tension are neglected, the fluid is incompressible and irrotational. The velocity field is provided by the  Euler equations, and the boundary behavior by the Bernoulli condition.  The starting point of their description was essentially the papers \cite{13a,10a}, where  a variant of the Boussinesq coupled system was derived for both first and second order in the small parameters $\alpha, \beta$ for which the well-posedness of the Cauchy problem was studied in \cite{10a,11a,Bous,Raz,Chen2}. 
From the first-order system they derived the mixed third-order KdV--BBM type equation
\begin{equation}\label{eq1a}
u_t+u_x+\frac 3 2 \alpha uu_x+\nu \beta u_{xxx}-\left(\frac 1 6 -\nu \right)\beta u_{xxt}=0,
\end{equation}
while from the  second-order Boussinesq system they derived the  unidirectional model,  a mixed fifth-order KdV--BBM type equation
\begin{equation}\label{eq2a}
u_{t}+u_{x}-\frac{\beta}{6} u_{xxt}+\delta_1 \beta^{2}\, u_{xxxxt}+\delta_2 \beta^{2}u_{xxxxx}+\frac{3\alpha}{2} uu_x
+\alpha \beta \gamma (u^2)_{xxx}-\frac{1}{12}\alpha \beta ({u_{x}}^{2})_x -\frac{1}{4}\alpha^2 (u^{3})_x=0.
\end{equation}
The small parameters $\alpha, \beta$  of both equations can be eliminated by reverting to
non-dimensional form,  denoted by tilde where $\tilde x=\frac{x}{\sqrt \beta}\, , \,\tilde t=\frac{t}{\sqrt \beta}$  and $u(x,t)=\frac 1 \alpha \tilde u(\tilde x, \tilde t) $ and by suppressing the tilde yields to 
\begin{equation}\label{eq1b}
u_t+u_x+\frac 3 2 uu_x+\nu u_{xxx}-\left(\frac 1 6 -\nu \right)u_{xxt}=0,
\end{equation}
for Eq.  (\ref{eq1a}), and to 
\begin{equation}\label{eq2b}
u_{t}+u_{x}-\frac{1}{6} u_{xxt}+\delta_1 u_{xxxxt}+\delta_2 u_{xxxxx}+\left(\frac{3}{4} u^2
+ \gamma (u^2)_{xx}-\frac{1}{12} {u_{x}}^{2} -\frac{1}{4}u^{3}\right)_x=0,
\end{equation}
for Eq. (\ref{eq2a}).  Notice that Eq. (\ref{eq1b}) is the BBM equation when $\nu=0$ and the KdV equation when $\nu=\frac 16 $. 
As it was explained in \cite{BCPS}  Eq. (\ref{eq1b})  can be derived by expanding the Dirichlet-Neumann operator \cite{31}, but this does not guarantee that the dispersion
relation  obtained fits the full dispersion to the order of the terms being kept, nor does
it guarantee that the resulting equation provides a well-posed problem. In  \cite{Amb1} and \cite{Amb2} this technique was applied to a deep-water
situation where the resulting system is  Hamiltonian, while the  initial-value
problem for it is ill-posed.

Higher-order versions of the KdV equation are not unique. Near-identity
transformations can be used to make the higher-order terms take any desired
form, as shown by the authors of  \cite{Kod,Hira,Fok,Cos}. Any second-order KdV equation can be mapped asymptotically to
integrable versions in the KdV hierarchy, such as the  KdV equation
itself. The same issue is also well-known for internal waves, see \cite{Grim}.


\section{Well-posedness and energy estimates}\label{sec2}
The local well-posedness of Eq. (\ref{eq1b}) is established in \S\ref{sec21} using a contraction mapping type argument combined with multilinear estimates.  The global well-posedness of Eq. (\ref{eq1b}) is established in \S\ref{sec22} in the spaces $H^s$, $s\geq 1$ which relies on the local results with energy type estimates.  The issue of  global well-posedness for the fifth-order Eq.  (\ref{eq2b}) is addressed in  \cite{BCPS}.  There, the assumption that $\gamma=\frac{1}{12}$ plays a crucial role in obtaining a conserved quantity which is used in the proof of global well-posedness for Eq. \eqref{eq1a}. Higher-order equations may not be Hamiltonian, but this is easily remedied through such a near-identity transformation. Thus, the restriction that 
$\gamma=\frac{1}{12}$  in   Eq. (\ref{eq2b})  in order for this equation to be Hamiltonian is
well-known and the remedy is to make such a transformation.  In \S \ref{sec23} we focus our attention on  the traveling wave reduction of Eq.  \eqref{eq1a}, and show that traveling wave solutions lie in $H^k$ for $k\geq 1 \in \mathbb N$ with no assumption on the parameter $\gamma$. For exact solutions on other types of   KdV--BBM  equations see \cite{Dut, Man,Nic,Kud} with all the  references therein.


\subsection{Local well-posedness}\label{sec21}
We are interested in the local well-posedness  associated to the Cauchy problem (\ref{eq1b}) for an initial profile $u(x,0)=u_0(x)\in H^s(\mathbb{R})$.  The local well-posedness for the fifth-order  Eq. (\ref{eq2b}) is established in \cite{BCPS}, and we follow these techniques in establishing the local well-posedness of the third-order  Eq. (\ref{eq1b}). 

Taking the Fourier transform of Eq. (\ref{eq1b}) about $x$ defined by 
$\widehat{ {u}(\xi,t)} = \mathcal F \{u(x,t)\}= \frac{1}{\sqrt{2 \pi}}\int_{\mathbb{R}}  u(x,t)\,e^{-i\xi x} \, dx$ yields  
\begin{equation}
\widehat{u}_t+i\xi\widehat{u}+\frac{3}{4}i\xi\widehat{u^2}-i\nu\xi^3\widehat{u}+\left(\frac{1}{6}-\nu\right)\xi^2\widehat{u_t}=0,
\end{equation}
which gives 
\begin{equation}
\left[\left(\frac{1}{6}-\nu\right)\xi^2+1\right]i\widehat{u_t}=\xi\left(1-\nu\xi^2\right)\widehat{u}+\frac{3}{4}\xi\widehat{u^2}.
\end{equation}
We assume that $\nu<\frac{1}{6}$, therefore $\left(\frac{1}{6}-\nu\right)\xi^2+1$ is positive and we have the equivalent equation
\begin{equation}
i\widehat{u_t}=\frac{\xi\left(1-\nu\xi^2\right)}{\left(\frac{1}{6}-\nu\right)\xi^2+1}\widehat{u}+\frac{3\xi}{4\left[\left(\frac{1}{6}-\nu\right)\xi^2+1\right]}\widehat{u^2}.
\end{equation}
Define the following Fourier multipliers $\phi(\partial_x)$, $\psi(\partial_x)$ based on their symbols
\begin{equation*}
\widehat{\phi(\partial_x)f}(\xi)=\phi(\xi)\widehat{f}(\xi),\;\;\widehat{\psi(\partial_x)f}(\xi)=\psi(\xi)\widehat{f}(\xi),
\end{equation*}
where $\phi(\xi)=\frac{\xi\left(1-\nu\xi^2\right)}{\left(\frac{1}{6}-\nu\right)\xi^2+1}$, $\psi(\xi)=\frac{3\xi}{4\left[\left(\frac{1}{6}-\nu\right)\xi^2+1\right]}.$  We can now reformulate the Cauchy problem in the following form
\begin{equation}\label{IVP} \begin{cases} 
      iu_t=\phi(\partial_x)u+\psi(\partial_x)u^2 \\
      u(x,0)=u_0(x), \\
   \end{cases}
\end{equation}
and we solve the associated linear problem
\begin{equation}\label{IVP2}  \begin{cases} 
      iu_t=\phi(\partial_x)u\\
      u(x,0)=u_0(x), \\
   \end{cases}
\end{equation}
where the solution is given by $u(t)=S(t)u_0$ with $\widehat{S(t)u_0}=e^{-i\phi(\xi)t}\widehat{u_0}$ defined by its Fourier transform.  Duhamel's principle gives us the integral form of the IVP~(\ref{IVP}):
\begin{equation*}
u(x,t)=S(t)u_0-i\int^t_0 S(t-s)\psi(\partial_x)u^2(x,s)ds.
\end{equation*}
The operator $S(t)$ is a unitary operator on $H^s$ for any $s\in\mathbb{R}$, $\left\|S(t)u_0\right\|_{H^s}=\left\|u_0\right\|_{H^s}$.  The local existence is established via a contraction mapping argument for the space $C\left(\left[0,T\right];H^s\right).$  We next need an estimate for the quadratic term in Eq. (\ref{IVP}).

\begin{lemma}
For any $s\geq 0$ there exists a constant $C=C_s$ such that the following inequality
\begin{align}\label{psi}
\left\|\psi(\partial_x)u^2\right\|_{H^s}\leq C_s\left\|u\right\|^2_{H^s},
\end{align}
holds for the operator $\psi(\partial_x)$.
\end{lemma}
\begin{proof} The proof relies on a bilinear estimate, see Lemma 3.1 in \cite{BCPS}.\end{proof}
With the above estimate, and $S(t)$ being a unitary operator we show that the operator 
\begin{align} 
\Psi u(x,t)=S(t)u_0-i\int^t_0 S(t-s)\psi(\partial_x)u^2(x,s)ds,
\end{align}
defines a contraction mapping on a closed ball $\mathcal{B}_r$ with radius $r>0$ centered at the origin in $C\left(\left[0,T\right];H^s\right)$.  First we show that $\Psi$ maps $\mathcal{B}_r$ to itself
\begin{align*}
\left\|\Psi u(x,t)\right\|_{H^s}\leq \left\|u_0\right\|_{H^s}+CT\left\|\psi(\partial_x)u^2\right\|_{C([0,T];H^s)}\leq \left\|u_0\right\|_{H^s}+CT\left\|u\right\|^2_{C([0,T];H^s)}.
\end{align*}
Since $u\in\mathcal{B}_r$ it follows that 
\begin{align*}
\left\|\Psi u(x,t)\right\|_{H^s}\leq \left\|u_0\right\|_{H^s}+CTr^2.
\end{align*}
In order for $\Psi$ to map $\mathcal{B}_r$ onto itself we choose $r=2\left\|u_0\right\|_{H^s}$ and  $T=\frac{1}{4Cr}$.

Next, we show that $\Psi$ is a contraction under the same choices of $r$ and $T$ by considering
\begin{align}
\Psi u(x,t)-\Psi v(x,t)=-i\int^t_0S(t-s)\psi(\partial_x)(u^2-v^2)(x,s)ds ,
\end{align}
which gives
\begin{align}
\left\|\Psi u(x,t)-\Psi v(x,t)\right\|_{C([0,T];H^s)}\leq 2rCT\left\|u-v\right\|_{C([0,T];H^s)}<\left\|u-v\right\|_{C([0,T];H^s)}.
\end{align}
We summarize all of the above in the following theorem:
\begin{theorem}\label{thm1}
For any $s\geq1$, and given $u_0\in H^s(\mathbb{R})$, there exists a time $T=T(\left\|u_0\right\|_{H^s})$, and a unique function $u\in C([0,T],H^s)$ which is a solution to the IVP (\ref{eq1b}) with initial data $u_0$.  The solution has continuous dependence on $u_0$ in $H^s(\mathbb{R})$.
\end{theorem}

\begin{corollary} The following properties of $u(x,t)$ follow from the proof of Theorem~\ref{thm1}.
\begin{enumerate}
	\item The maximal existence time $T=T_s$ of the solution satisfies 
	\begin{align}\label{Tmax}
	T\geq \tilde{T}=\frac{1}{4C_s\left\|u_0\right\|_{H^s}}.
	\end{align}
	\item We can bound the $H^s$ norm of the solution,
	\begin{align*}
	\left\|u(\cdot,t)\right\|_{H^s}\leq r=2\left\|u_0\right\|_{H^s}
	\end{align*}
	for all $t\in[0,\tilde{T}]$ where $\tilde{T}$ is defined in (\ref{Tmax}).
\end{enumerate}
\end{corollary}


\subsection{Global well-posedness}\label{sec22}
The goal of this subsection is to extend the local well-posedness  established in  \S \ref{sec21} for Eq. (\ref{eq1b}), and we derive energy estimates to obtain a global well-posedness result in $H^s(\mathbb{R})$ for $s\geq 1$. From the local theory the solution is only as smooth as the initial data, but one can make computations with smoother solutions and then pass to the limit of rougher initial data using the continuous dependence result by \cite{BK}.  Multiplying Eq. (\ref{eq1b}) by $u$ and integrating by parts over the spatial variable $x$ we obtain
\begin{equation}\label{eq3a}
\frac 1 2 \frac {d}{dt}\int_{\mathbb{R}} \left[ u^2+\left(\frac 1 6-\nu\right){u_x}^2\right] dx=-\left[\frac 1 2 u^2+\frac 1 2 u^3+\nu\left(uu_{xx}-\frac 1 2 {u_x}^2\right)\right]\Bigg|_{\mathbb{R}}.
\end{equation}
Therefore, if we assume that $u,\,u_x,\,u_{xx}\rightarrow 0$ as $|x|\rightarrow \infty$, the energy of the third-order equation which is a conserved quantity is 
\begin{equation}\label{eq4a}
E_3\big(u(\cdot,t)\big)=\frac 1 2 \int_{\mathbb{R}} \left[ u^2+\left(\frac 1 6-\nu\right){u_x}^2\right] dx= const,
\end{equation}
and for $\nu<\frac{1}{6}$ this is equivalent to the $H^1$ norm of $u$.  It is worth noting that a similar conserved quantity can be derived for Eq. (\ref{eq2b}), again by using parts to obtain 
\begin{equation}\label{eq3b}
\frac 1 2 \frac {d}{dt}\int_{\mathbb{R}} \left[ u^2+\frac 1 6 {u_x}^2+\delta_1 {u_{xx}}^2\right] dx=\left(\gamma-\frac {1}{12}\right)\int_{\mathbb{R}}{u_x}^3 dx,
\end{equation}
so the energy  is conserved when  $\gamma=\frac {1}{12}$, whereas the energy of the fifth-order equation is 
\begin{equation}\label{eq4q}
E_5\big(u(\cdot,t)\big)=\frac 1 2 \int_{\mathbb{R}} \left[ u^2+\frac 1 6 {u_x}^2+\delta_1 {u_{xx}}^2\right] dx=const. 
\end{equation}
For $\gamma=\frac{1}{12}$,  Eq. (\ref{eq4q}) is used in  \cite{BCPS} to establish global well-posedness for Eq.  (\ref{eq2b}) as we will see in the following lemma.
\begin{lemma}
There is a second conserved quantity
\begin{align}
\Phi(u)=\int_{\mathbb{R}}-\frac{1}{2}u^2-\frac{1}{2}u^3+\frac{\nu}{2}{u_{x}}^2\;dx,
\end{align} which can be used to express the system in Hamiltonian form
\begin{align*}
\frac{\partial }{\partial t}\nabla E_3(u)=\frac{\partial}{\partial x}\nabla \Phi(u),
\end{align*}
where $\nabla$ is the Euler derivative.
\end{lemma}
This leads us to the following global well-posedness result.
\begin{lemma}
Let $s\geq 1$ and suppose $\nu<\frac{1}{6}$.  Then the IVP for Eq. (\ref{eq1b}) is globally well-posed in $H^s(\mathbb{R})$.
\end{lemma}
\begin{proof}
The global well-posedness will be the consequence of an application of the local theory, and the estimate implied by the conserved quantity (\ref{eq4a}).  To obtain global well-posedness in $H^k$, $k\geq 2 \in \mathbb N $ we proceed by induction on $k$.  Assume $u_0\in H^2$, then there exists a $T=T(u_0)>0$ such that $u\in C([0,T];H^2)$.  Note that if one obtains an $H^k$ bound which is finite in time over finite time intervals then iterating would yield global bounds for the global solution $u$.  To achieve an $H^2$ bound, differentiate Eq. (\ref{eq1b}) and multiply by $u_x$ to obtain
\begin{align}
\frac{1}{2}\frac{d}{dt}\int_{\mathbb{R}} {u_{x}}^2+\left(\frac{1}{6}-\nu\right){u_{xx}}^2\;dx+\frac{3}{2}\int_{\mathbb{R}}{u_{x}}^3\;dx=0.
\end{align}
Applying the Gagliardo-Nirenberg inequality we have 
\begin{align*}
\left\|u_x\right\|^3_{L^3}\leq C\left\|u_{xx}\right\|^{2}_{L^2}\left\|u\right\|_{L^1},
\end{align*}
and from the energy (\ref{eq4a}), together with $\left\|u_x\right\|^2_{L^2}\leq E_3(u_0)$   yields
\begin{align*}
\left\|u_{xx}\right\|^{2}_{L^2}\leq \left\|u_x\right\|^2_{L^2}+C\int^t_0 \left\|u_{xx}\right\|^{2}_{L^2}\;dx.
\end{align*}
Thus, Gronwall's inequality gives
\begin{align*}
\left\|u_{xx}\right\|^{2}_{L^2}\leq \left\|u_x\right\|^2_{L^2}e^{Ct},
\end{align*}
from which the desired $H^2$-bound follows.  Using an inductive argument by assuming we have an $H^k$ bound and deriving an analogous estimate as above for $\left\|u\right\|_{H^{k+1}}$ we obtain an $H^{k+1}$ bound for $u$ assuming the initial datum lies in $H^{k+1}$. For regularity in the fractional spaces $H^s$ for $s\geq 1$, see the nonlinear interpolation theory in \cite{BS,BCW}.
\end{proof}
\subsection{$H^k$ bounds for traveling wave solutions}\label{sec23}
Notice that Eq.~(\ref{eq3b}) becomes an obstruction to establishing global well-posedness of Eq. (\ref{eq2b}) for $\gamma\neq \frac{1}{12}$, so instead of dealing with the global well-posedness of Eq. (\ref{eq2b}) for $\gamma=\frac{1}{12}$ we turn our focus to traveling wave solutions of  (\ref{eq2b}) without imposing a restriction on $\gamma$.  In this subsection we are interested in the smoothness of traveling wave solutions $u(\xi)=u(x-ct)$  of Eq. (\ref{eq2b}) while the derivation of the reduction is done in \S \ref{sec3}. Under that assumption  we derive $H^k$ bounds on $u$ where $k\geq 1 \in \mathbb N$.

In order to establish the $H^1$ bound we need the following identity for  $\int_\mathbb{R}{u_{\xi}}^3$.  Multiply Eq. (\ref{eq2}) first by $u_{\xi}$ and then integrate
\begin{eqnarray}\label{equ2}
\begin{array}{ll}
&\int_{\mathbb{R}}\left(\delta_2-c\delta_1\right)u_{\xi\xi\xi\xi}u_{\xi}+\frac{c}{6}u_{\xi\xi}u_{\xi}+\left(2\gamma-\frac{1}{12}\right){u_{\xi}}^3+2\gamma u u_{\xi}u_{\xi\xi}\;d\xi=\\  &=\int_{\mathbb{R}}\frac 1 4 u^3u_{\xi}-\frac{3}{4}u^2u_{\xi}-(1-c)uu_{\xi}+\mathcal B_1 u_\xi\;d\xi=\\ &=\frac{1}{16}u^4-\frac 1 4 u^3-\frac{1-c}{2}u^2+\mathcal{B}_1u\Big|_{\mathbb{R}}=0,
\end{array}
\end{eqnarray}
and using  parts we obtain 
\begin{equation}\label{conserv3}
\left(\gamma-\frac{1}{12}\right)\int_{\mathbb{R}}{u_{\xi}}^3\;d\xi=0.
\end{equation}
Appealing to Eq. (\ref{eq3b}) we see that for traveling wave solutions $u(x-ct)$ 
\begin{equation}\label{eq3ba}
\frac {d}{dt}\int_{\mathbb{R}} \left[ u^2+\frac 1 6 {u_x}^2+\delta_1 {u_{xx}}^2\right] dx=0.
\end{equation}
Therefore $\int_{\mathbb{R}}u^2+\frac{1}{6}{u_{\xi}}^2+\delta_1 {u_{\xi\xi}}^2\;d\xi=const$ yielding an $H^2$ bound for traveling wave solutions.  To prove regularity for $H^k$ for $k\geq 3 \in \mathbb N $ we use induction on $k$. Differentiate Eq. (\ref{eq2}), multiply by $u_{\xi}$ and integrate to obtain
\begin{equation}
\int_{\mathbb{R}}(\delta_2-c\delta_1){u_{\xi\xi\xi}}^2-\frac{c}{6}{u_{\xi\xi}}^2-2\gamma u{u_{\xi\xi}}^2\;d\xi=\int_{\mathbb{R}}\frac{3}{4}u^2{u_{\xi}}^2-\frac{3}{2}u{u_{\xi}}^2-(1-c){u_{\xi}}^2\;d\xi,
\end{equation}
which leads to
\begin{eqnarray}
\begin{array}{ll}\label{conserv4}
&(\delta_2-c\delta_1)\int_{\mathbb{R}}{u_{\xi\xi\xi}}^2\;d\xi=\\ 
&=\int_{\mathbb{R}}\frac{3}{4}u^2{u_{\xi}}^2-\frac{3}{2}u{u_{\xi}}^2-(1-c){u_{\xi}}^2+\frac{c}{6}{u_{\xi\xi}}^2+2\gamma u{u_{\xi\xi}}^2\;d\xi.
\end{array}
\end{eqnarray}
The above terms can be systematically bounded as follows
\begin{align}
&\int_{\mathbb{R}}u^2 {u_{\xi}}^2\;d\xi\leq \left\|u\right\|^2_{L^\infty}\left\|u_{\xi}\right\|^2_{L^2}<\infty\label{bd1},\\&
\int_{\mathbb{R}} u{u_{\xi}}^2\;d\xi\leq \left\|u\right\|_{L^\infty}\left\|u_{\xi}\right\|^2_{L^2}<\infty\label{bd2},\\&
\int_{\mathbb{R}} u{u_{\xi\xi}}^2\;d\xi\leq \left\|u\right\|_{L^\infty}\left\|u_{\xi\xi}\right\|^2_{L^2}<\infty\label{bd3}.
\end{align}
Combining Eqs. (\ref{bd1})-(\ref{bd3}) together with Eq.  (\ref{conserv4}) we have the desired $H^3$-bound. Using an inductive argument by assuming we have an $H^k$ bound and deriving an analogous estimate for $\left\|u\right\|_{H^{k+1}}$, we obtain an $H^{k+1}$ bound for $u$. To obtain regularity in the fractional spaces $H^s$ for $s\geq 2$, see the nonlinear interpolation theory in \cite{BS,BCW}.

Thus, we just proved:
\begin{theorem}
Assume $\delta_1>0$, let $\gamma\neq\frac{1}{12}$ and $k\in \mathbb{N}$.  If the solution $u(\xi)$ to the fifth-order KdV--BBM Eq. (\ref{eq1}) is bounded, then $u(\xi)$ lies in $H^k(\mathbb{R})$ when $k\in \mathbb{N}$.
\end{theorem}

\section{Traveling wave solutions}\label{sec3}
To find the traveling wave solutions of Eqs. (\ref{eq1b})-(\ref{eq2b})  we  use the well-known traveling wave anstaz  $u(\xi)=u(x-ct)$, where $c$ is the velocity of the unidirectional wave in the $x$ direction at time $t$ which reduces  to the third-order KdV--BBM
\begin{equation}\label{eq0}
(1-c)u_{\xi} +\frac 3 4 (u^2)_{\xi}+\left[\big(1-c\big)\nu+\frac c 6\right]u_{\xi \xi \xi}=0,
\end{equation}
and the fifth-order KdV--BBM, respectively
\begin{equation}\label{eq1}
(1-c)u_{\xi}+\frac{c}{6} u_{\xi \xi \xi}+(\delta_2-c\delta_1)u_{\xi \xi \xi\xi \xi} 
+\frac{3}{4} (u^2)_{\xi}+\gamma(u^2)_{\xi \xi \xi}-\frac{1}{12}({u_{\xi}}^2)_{\xi}-\frac 1 4 (u^3)_{\xi}=0.
\end{equation}
By integrating once both Eqs. (\ref{eq0})-(\ref{eq1}) we obtain the second and fourth order equations
\begin{equation}\label{eq1aa}
\left[\big(1-c\big)\nu+\frac c 6\right]u_{\xi \xi}=-\frac 34 u^2+(c-1)u+\mathcal A_1,
\end{equation}
\begin{equation}\label{eq2}
(\delta_2-c\delta_1)u_{\xi\xi\xi\xi}+\frac{c}{6}u_{\xi\xi}+\left(2 \gamma-\frac{1}{12}\right){u_{\xi}}^2+2 \gamma u u_{\xi\xi}=\frac 1 4 u^3-\frac 3 4 u^2 -(1-c)u+\mathcal B_1,
\end{equation} 
respectively, where  $\mathcal A_1, \mathcal B_1$ are  integration constants  which are zero if one assumes zero boundary conditions.   Similar fourth-order equations have been heavily studied in the literature in several physical contexts,  \cite{Kud,Chen1,Pol,Jin,Amin}, including capillary-gravity waves and elastic beams, and the full set
of solutions is very complicated with multi-hump solitary waves, generalized
solitary waves, and Jacobi elliptic  solutions. Next, we will  find elliptic solutions  of   Eqs. (\ref{eq1aa})-(\ref{eq2}) using the following lemma:
\begin{lemma}
Eqs. (\ref{eq1aa})-(\ref{eq2}) admit solutions which satisfy the elliptic equation 
\begin{equation}\label{eq3}
{u_{\xi}}^2=\sum_{i=0}^3 a_i u^i \equiv q_3(u),
\end{equation}
where $a_i$ are constants that depend on the parameters  $\nu, \gamma$, $\delta_1,\delta_2$, the boundary conditions $\mathcal A_0, \mathcal A_1,\mathcal B_1$ respectively, and the  wave speed $c$.  \\
\end{lemma}
\begin{proof} 
Differentiating once Eq. (\ref{eq3}) we obtain
\begin{equation}\label{eqa}
u_{\xi\xi}=\frac{1}{2}a_1+a_2 u+\frac{3}{2}a_3u^2.
\end{equation}
\begin{itemize}
\item[i)] For Eq. (\ref{eq1aa})  we define $\mu_1\equiv (1-c)\nu+\frac c 6$, and  identify the constants as
\begin{equation}\label{sys1}
\begin{array}{l}
a_0=\frac{\mathcal A_0}{\mu_1},\\
a_1=\frac{2 \mathcal A_1}{\mu_1},\\
a_2=\frac{c-1}{\mu_1},\\
a_3=-\frac{1}{2 \mu_1}.
\end{array}
\end{equation}
Thus, the solutions to Eq. \eqref{eq1aa} are found by solving
\begin{equation}\label{nen}
{u_{\xi}}^2=\frac{\mathcal A_0}{\mu_1}+\frac{2 \mathcal A_1}{\mu_1}u+\frac{c-1}{\mu_1}u^2-\frac{1}{2 \mu_1} u^3.
\end{equation}

\item[ii)] For Eq. (\ref{eq2}) differentiating twice Eq. (\ref{eqa}), and using Eqs  (\ref{eq3})-(\ref{eqa}) we  obtain  
\begin{equation}\label{eqb}
u_{\xi\xi \xi \xi}=3 a_0a_3+\frac 1 2 a_1a_2+\left(\frac 9 2 a_1 a_3+{a_2}^2\right)u+\frac{15}{2}a_2a_3u^2+\frac{15}{2}{a_3}^2u^3.
\end{equation}

 Using a balancing principle \cite{Nic}, we note that all the terms in Eq. (\ref{eq2}) are of degree $\leq 3$ which in conjunction with Eqs. (\ref{eq3})-(\ref{eqa}), (\ref {eqb}), yields  the nonlinear algebraic system

 \begin{equation}\label{sys2}
\begin{array}{l}
\frac{15}{2} \mu_2 {a_3}^2+(5 \gamma -\frac {1}{12})a_3=\frac 1 4\\
\frac{15}{2}\mu_2 a_2a_3+(4\gamma-\frac{1}{12})a_2+\frac c 4 a_3=-\frac 3 4\\
\mu_2{a_2}^2+\frac 9 2 \mu_2 a_1a_3+(3 \gamma-\frac{1}{12})a_1+\frac c 6 a_2=c-1\\
\frac 1 2 \mu_2a_1a_2+3 \mu_2a_0a_3+(2\gamma -\frac {1}{12})a_0+\frac{c}{12}a_1=\mathcal B_1,
\end{array}
\end{equation}
where  $\mu_2 \equiv \delta_2-c \delta_1$. Thus,  the solutions of Eq. (\ref{eq2}) are found by  solving the elliptic equation (\ref{eq3}) with constants $a_i$ from system \eqref{sys2}.

As is well known \cite{Man,Nic}, the general solution $u(\xi)$ of Eq. \eqref{eq3}  can be expressed in terms of Weierstrass  elliptic functions $\wp(\xi;g_2,g_3)$ which  satisfy the normal form
\begin{equation}\label{eq4}
{\wp_{\xi}}^2=4 \wp^3-g_2 \wp -g _3,
\end{equation} via the linear transformation (scale and shift)
\begin{equation}\label{eq5}
u(\xi)=\frac{4}{a_3}\wp(\xi;g_2,g_3)-\frac {a_2}{3a_3}.
\end{equation}
The germs (invariants) of the Weierstrass function are related to the coefficients of the cubic $q_3(u)$  and are given by
\begin{equation}\label{eq6}
\begin{array}{l}
g_2=\frac{{a_2}^2-3 a_1a_3}{12}=2(e_1^2+e_2^2+e_3^2),\\
g_3=\frac{9 a_1a_2a_3-27 a_0 {a_3}^2-2 {a_2}^3}{432}=4(e_1e_2e_3),
\end{array}
\end{equation}
and together with the modular discriminant 
\begin{equation}\label{eq7}
\Delta={g_2}^3-27 {g_3}^2=16(e_1-e_2)^2(e_1-e_3)^2(e_2-e_3)^2,
\end{equation}
are used to classify the solutions of Eq. \eqref{eq4}.  The constants  $e_i$ are the zeros of  the cubic polynomial 
\begin{equation*}\label{eq8}
p_3(t)=4t^3-g_2t-g_3,
\end{equation*}
and are related to the two periods $\omega_{1,2}$ of the $\wp$ function via $e_{{1,2}}=\wp{\left(\frac{\omega_{{1,2}}}{2}\right)}$, and $\omega_3=-(\omega_1+\omega_2)$.
\end{itemize}
\end{proof}

\begin{enumerate}
\item (Zero B.C.) When $a_0=0, a_1=0$ then $g_2=\frac{{a_2}^2}{12}, g_3=-\frac{{a_2}^3}{216}\Rightarrow \Delta  \equiv 0 $ which is the degenerate case. Hence, the Weierstrass solutions can be simplified moreover  since $\wp$ degenerates into trigonometric or hyperbolic functions, and that is due to the fact that $p_3(t)$ either has repeated root $e_i$ of multiplicity two ($m=2$) or three ($m=3$). We disregard the unphysical  case $m=3 \Rightarrow g_2=g_3=0$ for which $\wp(\xi;g_2,g_3)=\frac{1}{\xi^2}$.
When  $m=2$ and   $e_1=e_2>0$ then $e_3<0$, so $g_2>0, g_3<0$ and   we obtain hyperbolic bounded solutions, and if $e_2=e_3<0$ then $e_1>0$, so  $g_2>0, g_3>0$, and  we obtain trigonometric unbounded solutions.
\item (Nonzero B.C.) For  $ \Delta  \ne0 $ the general  solution to Eq. (\ref{eq3}) is given by the transformation   (\ref{eq5}). For a complete classification for which the $\wp$ functions can be simplified into the Equianharmonic and Lemniscatic case, see \cite{Abra}.
\end{enumerate}

\section{Results}

\subsection{Third order KdV-BBM equation}
\subsubsection{Solutions in terms of elementary functions (Zero B.C.)}
Using zero boundary conditions ($\mathcal A_0=\mathcal A_1=0$) in (\ref{sys1}) gives  the reduced system
 \begin{equation}\label{sys3}
\begin{array}{l}
a_0=a_1=0,\\
a_2=\frac{c-1}{\mu_1},\\
a_3=-\frac{1}{2 \mu_1},
\end{array}
\end{equation}
 which according to Eq. \eqref{nen}  leads to
\begin{equation}\label{eqaa}
{u_\xi}^2=\frac{c-1}{\mu_1} u^2-\frac{1}{2 \mu_1} u^3,
\end{equation}
with  solution
\begin{equation}\label{eqab}
u(\xi)=2(c-1)\mathrm{sech}^2\left(\frac 1 2\sqrt{\frac{c-1}{(1-c)\nu+\frac c 6}}\xi\right).
\end{equation}
For bounded solitons we require $\frac{c-1}{ (1-c)\nu+\frac c 6} >0$, while for unbounded periodic solutions  we require $\frac{c-1}{ (1-c)\nu+\frac c 6} <0$. All regions in the $(c,\nu)$ plane of solitary waves and periodic solutions are presented in Fig. \ref{figure1}, while in Fig. \ref{figure2} we present four traveling solutions which correspond  to each black dot of  Fig. \ref{figure1}.
\begin{figure}[ht!]
\centering
\includegraphics[width=0.4\textwidth]{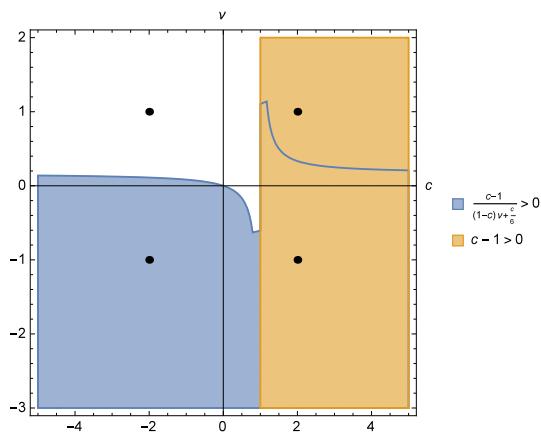}
\caption{\small{Region of existence of bounded $\frac{c-1}{ (1-c)\nu+\frac c 6}>0 $, and unbounded $\frac{c-1}{ (1-c)\nu+\frac c 6}<0 $. The solutions are bright ($c>1$), and dark ($c<1$). The 4 dots represent the values for which the solitons and periodic solutions are depicted  in Fig. \ref{figure2}.}}
\label{figure1}
\end{figure}
\begin{figure}[ht!]
\centering
\includegraphics[width=0.45\textwidth]{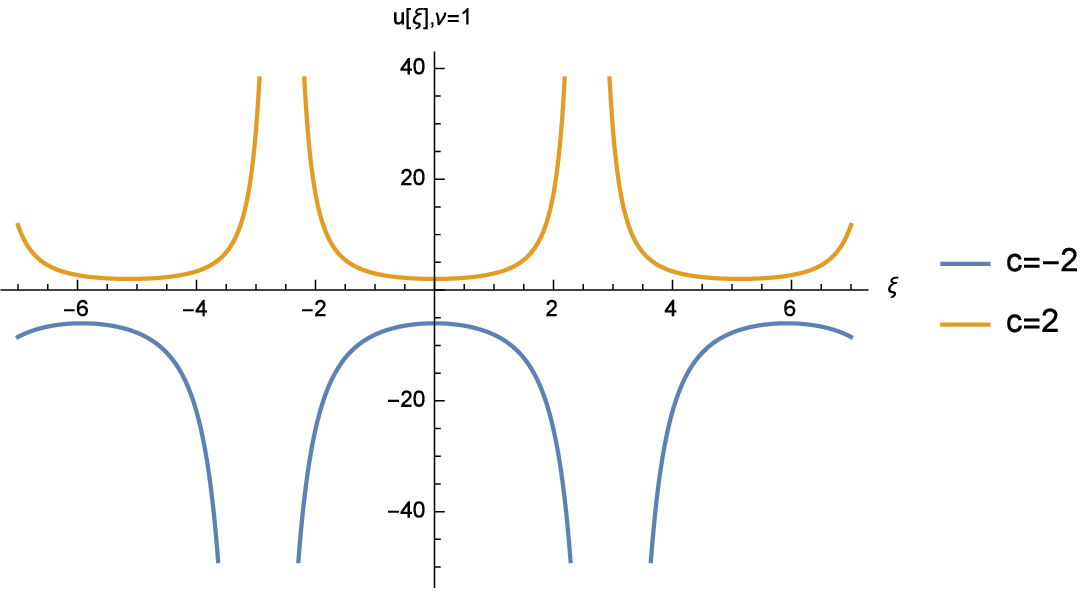}
\includegraphics[width=0.45\textwidth]{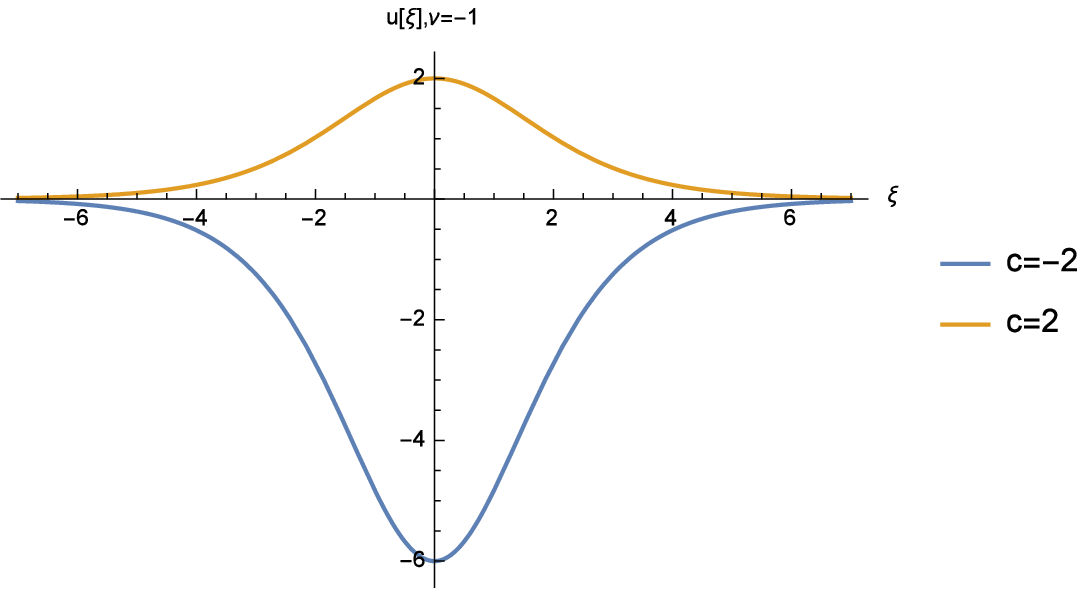}
\caption{\small{Unbounded (trigonometric)  solutions  with $\nu=1$ (left panel). Bounded solitons (hyperbolic)  bright ($c=2$) and dark solitons ($c=-2$)  for $\nu=-1$ (right panel) with  zero B.C. given by \eqref{eqab} for third-order KdV--BBM Eq. (\ref{eq0}).}}
\label{figure2}
\end{figure}
\subsubsection{Solutions in terms of elliptic functions (Nonzero B.C.)}
For nonzero boundary conditions then $\mathcal A_0 \ne 0,\, \mathcal A_1 \ne 0$, and  we  use the full system (\ref{sys1}) to  calculate the Weierstrass invariants using \eqref{eq6}
\begin{equation}\label{eq6a}
\begin{array}{l}
g_2=\frac{{a_2}^2-3 a_1a_3}{12}=\frac{3 {\mathcal A}_1+(-1+c)^2}{12 {\mu_1}^2}\\
g_3=\frac{9 a_1a_2a_3-27 a_0 {a_3}^2-2 {a_2}^3}{432}=-\frac{27{\mathcal A}_0+4(-1+c)\left[9{\mathcal A}_1+2(-1+c)^2\right]}{1278 {\mu_1}^3}.
\end{array}
\end{equation}
Then, using the transformation (\ref{eq5}), with constants from system ({\ref{sys1}}), and germs given by (\ref{eq6a}) the Weierstrass  solution to the third-order KdV--BBM Eq.  (\ref{eq0}) is
\begin{equation}\label{eqw}
u(\xi)=\frac 2 3 (-1+c)-8\left[(1-c)\nu+\frac c 6 \right]\wp\left[\xi;\frac{3 {\mathcal A}_1+(-1+c)^2}{12 \left[(1-c)\nu+\frac c 6 \right]^2},-\frac{27{\mathcal A}_0+4(-1+c)[9{\mathcal A}_1+2(-1+c)^2]}{1278 \left[(1-c)\nu+\frac c 6 \right]^3}\right].
\end{equation}
In  Fig. \ref{figure3} we present the Weierstrass solutions for nonzero boundary conditions  $\mathcal A_0=1, \mathcal A_1=1$ and coefficients $\nu=\pm 1$.
\begin{figure}[ht!]
\centering
\includegraphics[width=0.45\textwidth]{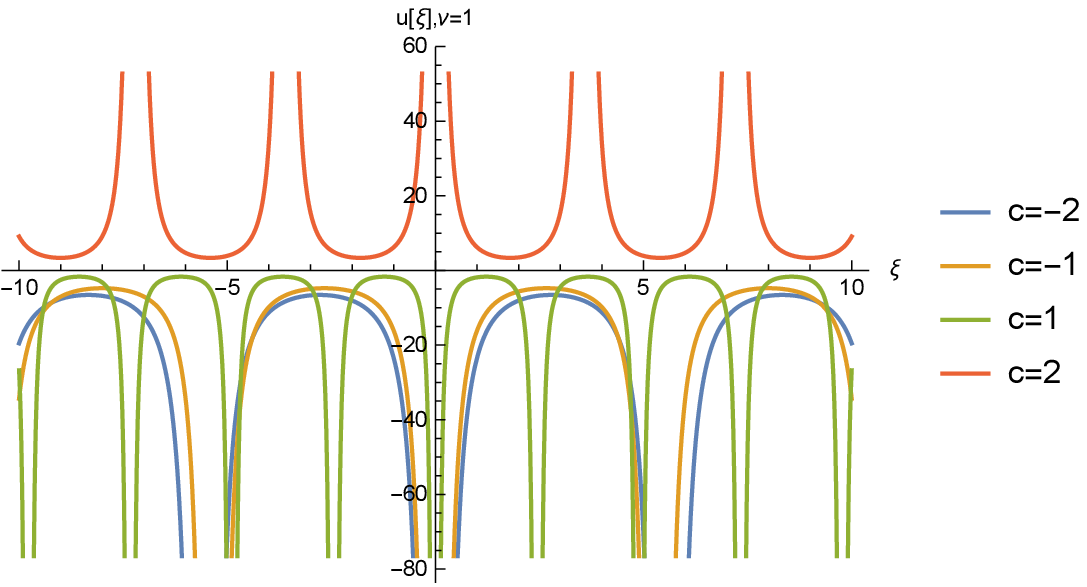}
\includegraphics[width=0.45\textwidth]{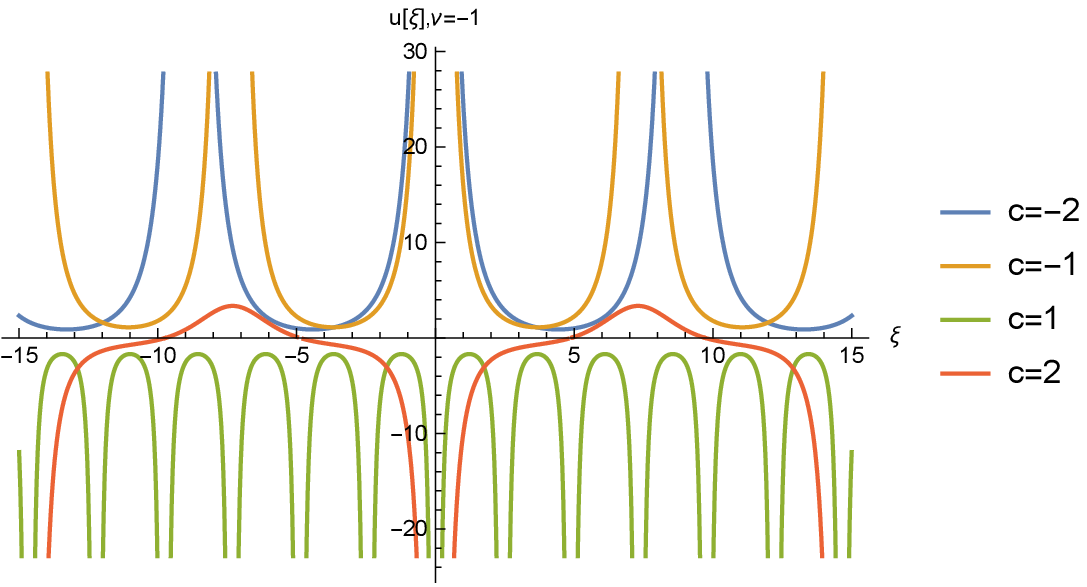}
\caption{\small{Weierstrass $\wp$ solutions  for nozero  B.C. $\mathcal A_0=1, \mathcal A_1=1$, $\nu=1$ (left panel), and $\nu=-1$ (right panel) given by \eqref{eqw} for third-order KdV--BBM Eq. (\ref{eq0}).}}
\label{figure3}
\end{figure}
\subsection{Fifth order equation}
\subsubsection{Solutions in terms of elementary functions (Zero B.C.)}
Proceeding is a similar manner, the system \eqref{sys2} with $a_0=a_1=0 \Rightarrow \mathcal B_1=0$ gives 
 \begin{equation}\label{sys6}
\begin{array}{l}
\frac{15}{2} \mu_2 {a_3}^2+(5 \gamma -\frac {1}{12})a_3=\frac 1 4\\
\frac{15}{2}\mu_2 a_2a_3+(4\gamma-\frac{1}{12})a_2+\frac c 4 a_3=-\frac 3 4\\
\mu_2{a_2}^2+\frac c 6 a_2=c-1.\\
\end{array}
\end{equation}
By finding $a_2$  from the last equation and $a_3$ from the first we obtain
\begin{equation}\label{sys7}
\begin{array}{l}
a_2=\frac{-c \pm \sqrt{c^2+144 \mu_2(c-1)}}{12 \mu_2},\\
a_3=\frac{1-60 \gamma \pm \sqrt{(1-60 \gamma)^2+1080 \mu_2}}{180 \mu_2}.\\
\end{array}
\end{equation}
For real solutions the constants $c,\mu_2, \gamma$ must be chosen such  that 
\begin{equation*}
(1-60 \gamma)^2+1080 \mu_2 \ge 0\, \wedge \, c^2+144 \mu_2(c-1) \ge 0,
\end{equation*} together with the level curves 
\begin{equation}\label{er}
h(\mu_2, c)=\frac{15}{2}\mu_2 a_2a_3+\left(4\gamma-\frac{1}{12}\right)a_2+\frac c 4 a_3+\frac 3 4=0.
\end{equation}

Thus, the solution of 
\begin{equation}\label{es1}
{u_{\xi}}^2=\frac{-c \pm \sqrt{c^2+144 \mu_2(c-1)}}{12 \mu_2}u^2+\frac{1-60 \gamma \pm \sqrt{(1-60 \gamma)^2+1080 \mu_2}}{180 \mu_2} u^3
\end{equation}
 is  given by 
 \begin{equation}\label{es2}
u(\xi)=\frac{15\left(-c\pm \sqrt{c^2+144 \mu_2(-1+c)}\right)}{60 \gamma-1\pm \sqrt{(1-60 \gamma)^2+1080  \mu_2}}\mathrm{sech}^2\left[\frac{\sqrt 3 \xi}{12}\sqrt{\frac{-c\pm\sqrt{c^2+144 \mu_2(-1+c)}}{ \mu_2}}\right].
\end{equation}

For the non Hamiltonian case, $\gamma\ne \frac {1}{ 12}$, we choose $\gamma=\frac 1 6$,  and varying $c$, we obtain unbounded solutions as well as both dark and bright solitons, which are presented in  Fig. \ref{figure4}. 
\begin{figure}[ht!]
\centering
\includegraphics[width=0.5\textwidth]{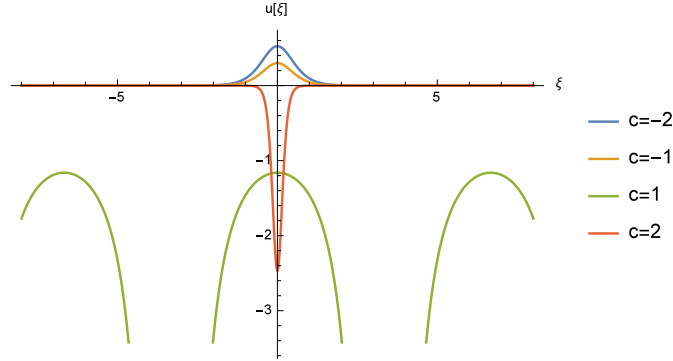}
\caption{\small{Evolution of bright ($c=-2,-1$) and dark solitons ($c=2$), together with unbounded periodic functions  ($c=1$) with  zero B.C. given by solution (\ref{es2}) for the fifth-order  KdV--BBM Eq. (\ref{eq1}),  and $\gamma =\frac{1}{6}$.}}
\label{figure4}
\end{figure}

For the Hamiltonian case, $\gamma=\frac{1}{12}$, we obtain $a_2=\frac{-c \pm \sqrt{c^2+144 \mu_2(c-1)}}{12 \mu_2}$ and $a_3=\frac{-2 \pm \sqrt{4+270 \mu_2}}{90 \mu_2}$. 
For real solutions the constants $c,\mu_2$ must be chosen such  that $4+270 \mu_2 \ge 0\, \wedge \, c^2+144 \mu_2(c-1) \ge 0$ which in turn gives either  $\mu_2 \in [-\frac{2}{135},0)$ or $\mu_2>0 \wedge c \in (-\infty,-72 \mu_2-12 \sqrt{\mu_2+36 \mu_2^2} ) \cup (-72 \mu_2+12 \sqrt{\mu_2+36 \mu_2^2},\infty)$. Bounded solitons which are presented in Fig. \ref{figure5}. 

\begin{figure}[ht!]
\centering
\includegraphics[width=0.45\textwidth]{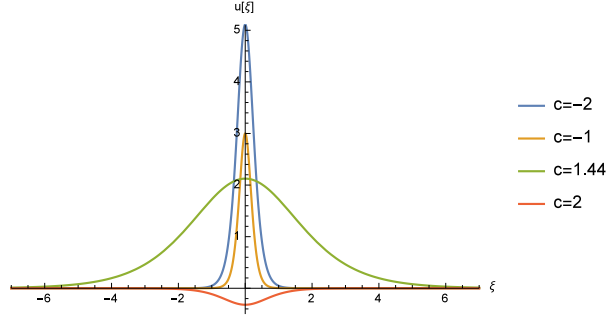}
\caption{\small{Evolution of bright ($c=-2,-1,1.44$) and dark ($c=2$) solitary waves with  zero B.C. given by solution (\ref{es2}) for the fifth-order  KdV--BBM Eq. (\ref{eq1}),  and $\gamma =\frac{1}{12}.$}}
\label{figure5}
\end{figure}

\subsubsection{Solutions in terms of elliptic functions (Nonzero B.C.)}
Finally, we use all four equations of system (\ref{sys2}) and we choose $\mathcal B_1=1$, $\gamma=\frac {1}{12}$, and constants $\delta_2, \delta_1$ such that $\mu_2=1$  to obtain the coefficients 
 \begin{equation}\label{sys7a}
\begin{array}{l}
a_0=\frac{180(1389041+848558\sqrt {274})+c\{95933536+87674598\sqrt {274}+c[-364106754+668133\sqrt {274}+c(19508\sqrt {274}-1623565)]\}}{2475480735},\\
a_1=\frac{c[-4086+9117\sqrt {274}+c(41 \sqrt {274}-145)]-360(33\sqrt {274}-59)}{124215},\\
a_2=\frac{c(\sqrt {274}-92)-90(\sqrt {274} -1)}{2730},\\
a_3=\frac{\sqrt {274}-2}{90}.
\end{array}
\end{equation}
We also  find  the germs using system  (\ref{eq6a}) to obtain the expressions
 \begin{equation}\label{sys8}
\begin{array}{l}
g_2=\frac{c[-56426+682 \sqrt{274}+3c(\sqrt{274}-53)]+98030-1180\sqrt{274}}{993720},\\
g_3=\frac{-80610056+1243238\sqrt{274}+c\{94841830-23958914\sqrt{274}+c[16001978-222019\sqrt{274}+c(91838-2125\sqrt{274})]\}}{101848350240}.
\end{array}
\end{equation}
Thus, the solution of the fifth-order KdV--BBM  Eq. (\ref{eq1}) is  obtained using the transformation (\ref{eq5}) with  coefficients from  system (\ref{sys7a}), and germs given by (\ref{sys8}). The  Weierstrass solutions are presented in  Fig. \ref{figure6} and are shown for $\gamma =\frac{1}{12}$ and $\mu_2=1$.

\begin{figure}[ht!]
\centering
\includegraphics[width=0.45\textwidth]{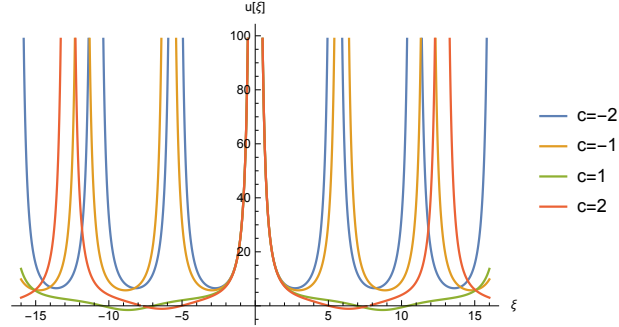}
\caption{\small{Weierstrass  $\wp$ solutions with  nonzero B.C. and $\gamma=\frac{1}{12}, \mu_2=1$ for the  fifth-order KdV--BBM  Eq. (\ref{eq1}).}}
\label{figure6}
\end{figure}


\section{Conclusion}
We have  provided  conditions for parameters for both third and fifth-order KdV--BBM   regularized long wave equations that yield  hyperbolic, trigonometric and elliptic solutions in terms of Weierstrass $\wp$ functions.  

For the third-order KdV--BBM equation an  analysis for the initial value problem  has been developed together with  a local well-posedness theory in relatively weak solution spaces. The global well-posedness is settled in the case of $\nu<\frac{1}{6}$.  For zero boundary conditions, and using a traveling wave ansatz we obtained  periodic (trigonometric) and solitary 
(hyperbolic) waves which are both  bright or dark with no restrictions on the parameter $\nu$ or the velocity $c$. For nonzero boundary conditions we showed a procedure using a balancing principle on how to construct   solutions in terms of Weierstrass $\wp$ functions. 

For the fifth-order KdV--BBM equation, motivated by the seemingly lack of a Hamiltonian structure for $\gamma\neq\frac{1}{12}$ we focused on the intrinsic properties of traveling wave solutions only.  We developed $H^k$ bounds for these solutions under the assumption that they  are not bounded.  We found that in the traveling wave variable $\xi$, $\gamma=\frac {1}{12}$ represents a   restriction  for which none of the constraint curves intersect the region of unbounded solitary waves, which shows that only dark or bright solitons and  no  unbounded solutions exist. When $\gamma \ne \frac{1}{12}$ since there is no restriction and the system is not Hamiltonian, solitons that are bright and dark coexist together with unbounded periodic solutions. For nonzero boundary conditions, due to the complexity of the nonlinear algebraic system of the coefficients of the elliptic equation, we have shown how Weierstrass solutions can also  be  constructed  for a particular set of parameters.  
\section*{References}
\bibliographystyle{elsarticle-num}
\bibliography{bibliography}
\end{document}